\documentclass[draft]{amsart}
\usepackage[matrix,arrow]{xy}
\usepackage{mathrsfs}
\usepackage{amssymb}

\numberwithin{equation}{section}

\newtheorem{thm}{Theorem}

\newtheorem{lem}[thm]{Lemma}

\theoremstyle{definition}

\theoremstyle{remark}
\newtheorem{rem}[thm]{Remark}

\DeclareMathOperator{\Ad}{Ad}

\newcommand{\C}{\mathbb{C}}

\newcommand{\R}{\mathbb{R}}

\newcommand{\Perv}{\mathrm{Perv}}
\newcommand{\Gm}{\mathbb{G}_{\mathit{m}}}
\newcommand{\HC}{\mathrm{HC}}
\newcommand{\pr}{\mathrm{pr}}
\title{A remark on the geometric Jacquet functor}
\subjclass[2010]{22E46, 14F05}
\author{Noriyuki Abe}
\author{Yoichi Mieda}
\address[Noriyuki Abe]{Graduate School of Mathematical Sciences, The University of Tokyo, 3-8-1 Komaba, Meguro, Tokyo, 153-8914 Japan}
\email{abenori@ms.u-tokyo.ac.jp}
\address[Yoichi Mieda]{Graduate School of Mathematics, Kyushu University, 744 Motooka, Nishi-ku, Fukuoka, 819-0395 Japan}
\email{mieda@math.kyushu-u.ac.jp}

\begin{document}
\begin{abstract}
We give an action of $N$ on the geometric Jacquet functor defined by Emerton-Nadler-Vilonen.
\end{abstract}

\maketitle

Let $G_\R$ be a reductive linear algebraic group over $\R$, $G_\R = K_\R A_\R N_\R$ an Iwasawa decomposition, and $M_\R$ the centralizer of $A_\R$ in $K_\R$.
Then $P_\R = M_\R A_\R N_\R$ is a Langlands decomposition of a minimal parabolic subgroup.
We use lower-case fraktur letters to denote the corresponding Lie algebras and omit the subscript ``$\R$'' to denote complexifications.
Fix a Cartan involution $\theta$ such that $K = \{g\in G\mid \theta(g) = g\}$.
For a $(\mathfrak{g},K)$-module $V$, the Jacquet module $J(V)$ of $V$ is defined by the space of $\mathfrak{n}$-finite vectors in $\varprojlim_{k\to\infty}V/\theta(\mathfrak{n})^kV$~\cite{MR562655}.

For simplicity, assume that $V$ has the same infinitesimal character as the trivial representation.
Denote the category of Harish-Chandra modules with the same infinitesimal characters as the trivial representation by $\HC_\rho$.
Let $X$ be the flag variety of $G$, $\Perv_K(X)$ the category of $K$-equivariant perverse sheaves on $X$.
By the Beilinson-Bernstein correspondence and the Riemann-Hilbert correspondence, we have the localization functor $\Delta\colon \HC_\rho\to\Perv_K(X)$.
Emerton-Nadler-Vilonen gave a geometric description of $J(V)$ by the following way~\cite{MR2096674}.
Fix a cocharacter $\nu\colon\Gm\to A$ which is positive on the roots in $\mathfrak{n}$.
Define $a\colon \Gm\times X\to X$ by $a(t,x) = \nu(t)x$.
Consider the following diagram,
\[
	X\simeq \{0\}\times X\to \mathbb{A}^1\times X\leftarrow \Gm\times X\xrightarrow{a} X.
\]
Let $R\psi$ be the nearby cycle functor with respect to $\mathbb{A}^1\times X \to \mathbb{A}^1$.
Then the geometric Jacquet functor $\Psi$ is defined by 
\[
	\Psi(\mathscr{F}) = R\psi(a^*(\mathscr{F})).
\]
\begin{thm}[Emerton-Nadler-Vilonen~{\cite[Theorem~1.1]{MR2096674}}]
We have $\Delta\circ J\simeq \Psi\circ \Delta\colon \HC_\rho\to \Perv_K(X)$.
\end{thm}
It is easy to see that $J(V)$ is a $(\mathfrak{g},N)$-module for a $(\mathfrak{g},K)$-module $V$.
Hence $\Psi(\mathscr{F})$ is $N$-equivariant for $\mathscr{F}\in \Perv_K(X)$. (See also \cite[Remark~1.3]{MR2096674}.)

In this paper, we give the action of $N$ on $\Psi(\mathscr{F})$ in a geometric way.
Roughly speaking, this action is given by the ``limit'' of the action of $K$.

We use the following lemma.
\begin{lem}\label{lem:limit of action}
Let $\mathcal{X}$ be a scheme of finite type over $\mathbb{A}^1$, $\mathcal{X}^0$ (resp.~$\mathcal{X}_0$) the inverse image of $\Gm$ (resp.~$\{0\}$), and $\mathcal{G}$ a smooth group scheme over $\mathbb{A}^1$.
If $\mathscr{F}^0$ is a $\mathcal{G}^0$-equivariant perverse sheaf on $\mathcal{X}^0$, then $R\psi(\mathscr{F}^0)$ is $\mathcal{G}_0$-equivariant.
\end{lem}
\begin{proof}
Define $m\colon \mathcal{G}\times_{\mathbb{A}^1}\mathcal{X}\to\mathcal{X}$ by $m(g,x) = gx$.
Then $m$ is a smooth morphism.
Let $m^0\colon \mathcal{G}^0\times_{\Gm}\mathcal{X}^0\to \mathcal{X}^0$ and $m_0\colon \mathcal{G}_0\times \mathcal{X}_0\to \mathcal{X}_0$ be its restrictions.
Since $\mathscr{F}^0$ is $\mathcal{G}^0$-equivariant, we have an isomorphism $(m^0)^*(\mathscr{F}^0)\simeq \pr_2^*(\mathscr{F}^0)$.
Hence we have $R\psi (m^0)^*(\mathscr{F}^0)\simeq R\psi\pr_2^*(\mathscr{F}^0)$.
By the smooth base change theorem, $m_0^*R\psi(\mathscr{F}^0)\simeq \pr_2^*R\psi(\mathscr{F}^0)$~\cite[Expos\'e XIII]{MR0354657}.
Hence $R\psi(\mathscr{F}^0)$ is $\mathcal{G}_0$-equivariant.
\end{proof}

Set
\[
	\mathcal{K}^0 = \{(t,k)\in \Gm\times G\mid k\in\Ad(\nu(t)^{-1})(K)\}\subset \mathbb{A}^1\times G.
\]
Let $\mathcal{K}$ be the closure of $\mathcal{K}^0$ in $\mathbb{A}^1\times G$.
It is a closed sub-group scheme of $\mathbb{A}^1\times G$.
Then $\mathcal{K}$ is flat over $\mathbb{A}^1$.
Since each fiber of $\mathcal{K}\to \mathbb{A}^1$ is a group scheme of finite type over $\C$, it is reducible~\cite{MR0206005}.
Hence it is smooth.
Therefore, $\mathcal{K}$ is smooth over $\mathbb{A}^1$.

Let $\Sigma$ be the restricted root system of $(\mathfrak{g},\mathfrak{a})$, $\Sigma^+$ the positive system corresponding to $\mathfrak{n}$, and $\mathfrak{g}_\alpha$ the restricted root space for $\alpha\in\Sigma$.
Then $\mathfrak{k}$ is spanned by $\mathfrak{m}$ and $\{X + \theta(X)\mid X\in \mathfrak{g}_\alpha,\alpha\in\Sigma^+\}$.
Since 
\[
	\Ad(\nu(t)^{-1})(X + \theta(X)) = t^{-\langle \nu,\alpha\rangle}(X + t^{2\langle \nu,\alpha\rangle}\theta(X))
\]
for $X\in \mathfrak{g}_\alpha$, the Lie algebra of $\Ad(\nu(t)^{-1})(K)$ is spanned by 
\[
	\text{$\mathfrak{m}$ and $\{X + t^{2\langle \nu,\alpha\rangle}\theta(X)\mid X\in \mathfrak{g}_\alpha, \alpha\in\Sigma^+\}$.}
\]
Hence the neutral component of $\mathcal{K}_0$ is $M^\circ N$ where $M^\circ$ is the neutral component of $M$.
Since $MK^\circ = K$ and $\Ad(\nu(t)^{-1})(M) = M$, we have $\mathcal{K}_0 = MM^\circ N = MN$.

Define $\widetilde{a}\colon \mathcal{K}^0\times_{\Gm}(\Gm\times X)\to K\times X$ (resp.~$\widetilde{m}\colon \mathcal{K}^0\times_{\Gm}(\Gm\times X)\to \Gm\times X$, $m\colon K\times X\to X$) by $\widetilde{a}((t,k),(t,x)) = (\Ad(\nu(t))k,\nu(t)x)$ (resp.~$\widetilde{m}((t,k),(t,x)) = (t,kx)$, $m(k,x) = kx$).
Then we have the following commutative diagrams
\[
\xymatrix{
	\mathcal{K}^0\times_{\Gm}(\Gm\times X) \ar[r]^(0.6){\widetilde{m}}\ar[d]_{\widetilde{a}} & \Gm \times X\ar[d]^a\\
	K\times X\ar[r]^(0.54)m & X,
}
\quad
\xymatrix{
	\mathcal{K}^0\times_{\Gm}(\Gm\times X) \ar[r]^(0.6){\pr_2}\ar[d]_{\widetilde{a}} & \Gm \times X\ar[d]^a\\
	K\times X\ar[r]^(0.54){\pr_2} & X.
}
\]
Let $\mathscr{F}\in \Perv_K(X)$.
Then $m^*\mathscr{F}\simeq \pr_2^*\mathscr{F}$.
Hence we get $\widetilde{a}^*m^*\mathscr{F}\simeq \widetilde{a}^*\pr_2^*\mathscr{F}$.
By the above diagrams, $\widetilde{m}^*a^*\mathscr{F}\simeq \pr_2^*a^*\mathscr{F}$.
Therefore $a^*(\mathscr{F})$ is $\mathcal{K}^0$-equivariant.
By Lemma~\ref{lem:limit of action}, $\Psi(\mathscr{F}) = R\psi(a^*(\mathscr{F}))$ is $\mathcal{K}_0 = MN$-equivariant .

\begin{rem}
Let $\Gamma$ be the quasi-inverse functor of $\Delta$.
Then $N$ acts on $\Gamma(\Psi(\mathscr{F}))$.
Moreover, this becomes a $(\mathfrak{g},N)$-module~\cite[9.1.1]{MR1021510}, namely, the infinitesimal action of $N$ coincides with the action of $\mathfrak{n}\subset \mathfrak{g}$.
We also have that $J(\Gamma(\mathscr{F}))$ is a $(\mathfrak{g},N)$-module.
Hence both $N$-actions have the same infinitesimal actions.
Since the action of $N$ is determined by its infinitesimal action, the $N$-action we defined above coincides with the $N$-action on $J(\Gamma(\mathscr{F}))$.
\end{rem}

%\bibliography{../../bunken,../../book}
%\bibliographystyle{my_amsalpha}

\end{document}